\documentclass[12pt]{article}

\usepackage[T1]{fontenc}
\usepackage{avant}
\usepackage[cp1250]{inputenc}
\usepackage{amsmath,amssymb,amsthm}
\allowdisplaybreaks[2]
\usepackage{amscd}
\usepackage{mathtools} 
\usepackage{stmaryrd} 
\usepackage[ruled,linesnumbered,norelsize]{algorithm2e} 
\usepackage{longtable} 
\usepackage{url} 
\usepackage{authblk} 

\usepackage{tikz}
\usepackage[all]{xy}
\newtheorem{theorem}{Theorem}
\newtheorem{proposition}{Proposition}

\newtheorem{definition}{Definition}

\newtheorem{remark}{Remark}

\newtheorem{problem}{Problem}

\title{On non-positive Weyl connections on  Lie groups}
\author{Tomasz Dudek and Aleksy Tralle}

\begin{document}
\maketitle{}

\section{Introduction}
Let there be given a Riemannian manifold $(M,\langle -,-\rangle)$. Let $\nabla$ denote the Levi-Civita coonection.
By definition, a Weyl connection is given by the formula
$$\hat\nabla_XY=\nabla_XY+\varphi(Y)X+\varphi(X)Y-\langle X,Y\rangle E,$$
where $\varphi$ is a one-form $\varphi(X)=\langle X,E\rangle$, $X,Y$ are arbitrary vector ields, $E$ is fixed.
Given a Weyl connection $\widehat{\nabla}$ one considers the curvature tensor
$$\widehat{R}(X,Y)=\widehat{\nabla}_X\widehat{\nabla}_Y-\widehat{\nabla}_Y\widehat{\nabla}_X-\widehat{\nabla}_{[X,Y]}$$ 
and takes the {\it anti-symmetric} part $\widehat{R}_a$ of the tensor $\widehat{R}$. This tensor is called {\it the directional curvature}.
\begin{definition} The Weyl sectional curvature {\rm is defined as}
$$\widehat{K}(\Pi)=\langle \widehat{R}_a(X,Y)Y,X\rangle,\Pi=
\mathrm{Span}_{\mathbb{R}}(X,Y),\,|X|=|Y|=1, X\perp Y.$$
\end{definition}
\begin{proposition}[\cite{TW}] The sign of $\widehat{K}(\Pi)$ is well defined.
\end{proposition}
There are several reasons for an interest in studying such connections and their curvature properties. First, we mention the applications in dynamical systems. In \cite{W} the notion of W-flow on the unit tangent bundle to $(M,\langle -,-\rangle)$ was introduced. The orbits of such flow are geodesics of $\widehat{\nabla}$. An interesting theorem in \cite{W}, Theorem 5.1  says that (in particular) on a Riemannian manifold with a Weyl connection whose {\it Weyl curvature} is negative everywhere the W-flow $\Phi$ has dominated splitting, that is, the tangent bundle continuously splits as 
$TM=E\oplus F$ in a way that the subbundles $E$ and $F$ are invariant with respect to the differential $\Phi_*: TM\rightarrow TM$ and  have the following property: 
$$||\Phi_*)^N_xv||\leq \frac{1}{2}||(\Phi_*)_x(w)||,\qquad \forall v\in E_x, 
\quad w\in F_x.$$ 
It follows that W-flows on such compact  manifolds  have a kind of "hyperbolic" properties akin to the Anosov property. 

The second reason comes from the theory of Gaussian thermostats \cite{H} which yields interesting models in statistical physics.

 It follows that a problem of  description of Riemannian manifolds with negative Weyl connections is intriguing, but difficult, and still not much is known. In particular, the following problem was posed in \cite{W}.         

\begin{problem}\label{probl-gen}
 Construct a Weyl structure $(M,\langle -,-\rangle,\widehat{\nabla})$ with negative Weyl curvature on a manifold which does not carry a Riemannian metric of negative sectional curvature.
\end{problem}
It is quite natural to begin with the case of homogeneous spaces $G/H$ of Lie groups and invariant Weyl connections, since in this case one can use the power of Lie theory. 
\begin{problem}\label{probl:inv}
Describe invariant Weyl connections on homogeneous spaces and solve Problem \ref{probl-gen} in the homogeneous case. 
\end{problem}
The first results in this direction were obtained in \cite{TW}. In this work, the authors propose a way to look for Weyl connections with {\it non-positive} Weyl curvature by rescaling the vector field $E$ which determines $\widehat{\nabla}$.   
\begin{definition} {\rm By definition a {\it stretching} of $\widehat{\nabla}$ is a one-parameter  family $\widehat{\nabla}_t$ of Weyl connections}
$$(\widehat{\nabla}_t)_XY=\nabla_XY+t\varphi(X)Y+t\varphi(Y)X-t\langle X,Y\rangle E.$$
{\rm A connection $\widehat{\nabla}$ is called stretched non-positive (SNP) if there exists a stretching such that $\widehat{\nabla}_t$ is non-positive for sufficiently large $t$.}
\end{definition}     
Since the problem (even for SNP) is difficult for arbitrary homogeneous spaces, it is natural to begin with invariant Riemannian metrics on a Lie group $G$  and consider  Lie groups endowed with left-invariant Riemannian metrics and invariant Weyl connections determined by left-invariant vector fields $E\in\mathfrak{g}$ (that is, vectors in the corresponding Lie algebra $\mathfrak{g}$). In \cite {TW} the authors analyzed the case of 3-dimensional Lie groups and found an (isolated) example of non-positive on a solvable Lie group SOL. They conjecture that, apart from the flat case, in dimension three only SOL carries such connections. We confirm this in Theorem \ref{thm:main} in the case of completely solvable Lie groups.  In dimension 4,  SNP connections were found on some 4-dimensional solvable Lie groups, and the structure of solvable Lie groups endowed with SNP connctions determined by parallel vector fields $E$ was described in \cite{BJT}. In this note we continue this line of research. In particular we prove the following. 
\begin{theorem}\label{thm:main}Let $G$ be a simply connected, non-abelian, three-dimensional unimodular completely solvable Lie group. Then $G$ admits a left-invariant Riemannian metric and a left-invariant Weyl connection of non-positive sectional Weyl curvature if and only if $G$ is the group SOL. None of these connections are SNP.
\end{theorem}
\begin{remark}{\rm The Lie group $G=\tilde{E}_0(2)$ which is a universal cover of the group  of proper motions of the euclidean plane does have an SNP connection for the flat metric (see \cite{TW}). The corresponding Lie algebra is not completely solvable. Our method yields the same result (see the comment at the end of this article).}
\end{remark}
Our approach is based on the following:  we express the non-positivity condition in terms of the Nomizu function and simplify the calculations using the classification of the invariant Riemannian metrics on completely solvable Lie groups up to Lie algebra isomorphisms \cite{KHL}.  We believe that this approach can be extended to  arbitrary completely solvable Lie groups.   
  
\section{Invariant connections on homogeneous\\ spaces of Lie groups}
We will follow more modern exposition of the classical Nomizu's theorem in \cite{DGP}. In this section $\nabla$ denotes arbitrary affine connection (not necessarily Levi-Civita or Weyl). Let $M$ be a smooth manifold with a transitive action of a Lie group $G$. For any $g\in G$ consider the left translation $\tau(g)(p)=g\cdot p$ for all $p\in $. Let $X$ denote a vector field on $M$. Then one can define vector field $\tau_g(X)$ by the formula 
$$(\tau_{g}(X))_p=(\tau_g)_*(X_{g^{-1}\cdot p}).$$
An invariant affine connection on $M$ is an affine connection $\nabla$ which satisfies 
$$\tau_g(\nabla_XY)=\nabla_{\tau_g(X)}\tau_g(Y).$$
Recall that for every affine connection $\nabla$ and an arbitrary vector field $Z$ there is a $(1,1)$-tensor field on $M$ given by
$$L_Z^{\nabla}X=[Z,X]-\nabla_ZX.$$
Since $G$ acts transitively on $M$ and $\nabla$ is $G$-invariant, we get a commutative diagram
$$
\CD
T_oM @>{L^{\nabla}_{(Ad g^{-1}A)^+}}>> T_oM\\
@A{ (\tau_{g^{-1}})_*}AA  @A{(\tau_{g^{-1})_*}}AA\\
T_oM @>>{L^{\nabla}_{A^+}}> T_oM
\endCD
$$
Let $G/H$ be a homogeneus space. We will call it reductive with a reductive decomposition $\mathfrak{g}=\mathfrak{h}\oplus\mathfrak{m}$, if $Ad(H)(\mathfrak{m})\subset\mathfrak{m}$. Note that we fix the reductive decomposition, since it is not unique and the same manifold $M$ may have different presentations as a homogeneous space $G/H$.  For the given reductive decomposition one can identify $\mathfrak{m}$ with $T_oM$. In this case  the above diagram can be rewritten as 
$$
\CD
T_oM @>{L^{\nabla}_{(Ad g^{-1}A)^+}}>> T_oM\\
@A{\pi_*}AA @A{\pi_*}AA\\
\mathfrak{m} @>>{\alpha_{\nabla}(A,-)}> \mathfrak{m}
\endCD
$$
This formula shows that $\nabla$ is recovered from the bilinear $Ad(H)$-equivariant map 
$$\alpha_{\nabla}:\mathfrak{m}\times\mathfrak{m}\rightarrow \mathfrak{m}.$$
For more details we refer to \cite{DGP}.
Finally we will formulate the above considerations as follows.
\begin{theorem}[Nomizu, \cite{N}] Let $G/H$ be a reductive homogeneous space with a fixed reductive decomposition $\mathfrak{g}=\mathfrak{h}\oplus\mathfrak{m}$. There is a bijective correspondence between the set of (left) $G$-invariant connections $\nabla$ on $G/H$ and the vector space of $Ad(H)$-equivariant  bilinear maps 
$$\alpha:\mathfrak{m}\times\mathfrak{m}\rightarrow\mathfrak{m}.$$
\end{theorem}
\noindent In the sequel we will call $\alpha$ the {\it Nomizu function}.

\section{Weyl curvature of left-invariant Weyl connections}
Let $(G,\langle -,-\rangle)$ be a Lie group endowed with an invariant Riemannian metric.
\begin{proposition}[\cite{KN}, Chapter X, Theorem 3.3]\label{prop:KN-LC} The Nomizu function $\alpha$ of the Levi-Civita connection on $(G,\langle,-,\rangle)$ is given by the formula
$$\alpha(X,Y)=\frac{1}{2}[X,Y]+U(X,Y)$$
where $U(X,Y)$ is determined by
$$2\langle U(X,Y),Z\rangle=\langle X,[Z,Y]\rangle+\langle [Z,X],Y\rangle.$$
\end{proposition}
\begin{proposition}\label{prop:hat-k} Let $(G,\langle -,-\rangle)$ be a Lie group endowed with a left-invariant Weyl connection determined by $E\in\mathfrak{g}$. The Weyl curvature of this connection is given by the formula
 \begin{align}\label{eq:general-weyl}
\widehat{K}(X,Y) &= K(X,Y)  + \langle {Y},{E}\rangle^2 + \langle {X},{E}\rangle^2 - \langle{E},{E}\rangle\notag\\ &\quad - \langle{Y},{\alpha(Y,E)}\rangle  - \langle{\alpha(X,E)},{X}\rangle,
\end{align}
for any unit orthogonal $X,Y$.
\end{proposition}
\begin{proof} Use the following formula for the Weyl curvature \cite{W} (Proposition 3.2):
$$\widehat{R}_a(X,Y)Z=R(X,Y)Z+\langle Z,E\rangle(\langle Y,E\rangle X-\langle X,E\rangle Y)$$
$$+(\langle Z,Y\rangle \langle X,E\rangle-\langle Z,X\rangle\langle Y,E\rangle)E+E^2(\langle Z,X\rangle Y-\langle Z,Y\rangle X)$$
$$+\langle Z,\nabla_XE\rangle Y-\langle Z,\nabla_YE\rangle X+\langle Z,X\rangle\nabla_YE-\langle Z,Y\rangle\nabla_XE$$ 
  together with assumptions $Z=Y$, $\langle X,X\rangle=\langle Y,Y\rangle=1$, $\langle X,Y\rangle=0$ and the formula $(\nabla_UV)_o=\alpha(U,V)$.
\end{proof}

\section{Unimodular solvable Lie groups in dimension three with non-positive Weyl connections}

The purpose of this section is to prove Theorem \ref{thm:main}. Among the non-abelian unimodular solvable Lie algebras in dimension three, only NIL and SOL are completely solvable. The abelian case is immediate: the metric is flat and every left-invariant vector field is parallel; its Weyl sectional curvature is $-\langle E,N\rangle^2$ for a unit normal $N$, so every such field is SNP.
\subsection{Left-invariant Riemannian metrics on completely solvable Lie groups}
\begin{definition}[\cite{A}]\label{def:metric-lie-algebra} {\rm  A Lie algebra $\mathfrak{g}$ together with a fixed scalar product $\langle-,-\rangle$ is called a metric Lie algebra. Two metric Lie algebras $\mathfrak{g}$ and $\mathfrak{g}'$ are called isomorphic if there exists an isometry  $\varphi:\mathfrak{g}\rightarrow\mathfrak{g}'$ which  is simultaneously  an isomorphism of Lie algebras.} 
\end{definition}
As an example we consider solvable Lie algebras which have the following property: for any $X\in\mathfrak{g}$ the operator $\mathrm{ad}_X$ has only real eigenvalues.  
\begin{theorem}[\cite{A}]\label{thm:aleks} Isometric completely solvable metric Lie algebras are isomorphic.
\end{theorem}

\subsection{Classification of left-invariant Riemannian metris and sectional curvaures on 3-dimensional Lie groups}
We explain the classification of left-invariant metrics on Lie groups up to isomorphism. Let $G$ be a simply connected Lie group and $l_g: G\rightarrow G$ the left translation determined by $g$. A left invariant Riemannian metric on $G$ yields a scalar product on the Lie algebra $\mathfrak{g}$. Conversely, if $\langle -,-\rangle$ is a scalar product on $\mathfrak{g}$, the formula
$$\langle v,e\rangle_g=\langle (l_{g^{-1}})_*v,(l_{g^{-1}})_*w\rangle_e$$ 
defines a left-invariant Riemannian metric on $G$. The group $\mathrm{Aut}(\mathfrak{g})$ acts on the set $\mathcal{M}$ of all scalar products on $\mathfrak{g}$ by the formula
$$\varphi\cdot\langle v,w\rangle=\langle\varphi^{-1}v,\varphi^{-1}w\rangle.$$
Note that if one defines left-invariant Riemannian metrics on $G$ by the scalar products $\langle -,-\rangle$ and $\varphi\cdot\langle -,-\rangle$  lying in the same $\mathrm{Aut}(\mathfrak{g})$-orbit, then $(G,\langle -,-\rangle)$ and $(G,\varphi\cdot\langle v,w\rangle)$ will be isometric. Thus, a description of the space $\mathcal{M}/\mathrm{Aut}(\mathfrak{g})$ is a classification of left-invariant metrics on simply connected Lie groups {\it up to isomorphism}. In general not every isometry of $G$ is determined by an automorphism of $\mathfrak{g}$, thus this classification is finer then the classification up to isometry. However, it is  sufficient for our purposes. In our calculations we use this clasification in the form \cite{KHL}. Here is the sum up of this description. Let $[g]=(g_{ij})$ and $[g']=[g'_{ij}]$ be positive definite matrices representing the scalar products on $\mathfrak{g}$. We write $[g]\simeq [g']$ if 
$$[g']=[\varphi]^t[g][\varphi].$$
In this form, the classification of left-invariant Riemannian metrics on $G$ amounts to finding matrices representing a scalar product in each $\mathrm{Aut}(\mathfrak{g})$-orbit. Let us write down these matrices in the case of solvable unimodular 3-dimensional Lie groups.
It is well known that such groups are
\begin{enumerate}
\item abelian $\mathbb{R}^3$ (which is clear and not interesting from our point of view),
\item the Heisenberg group NIL whose Lie algebra is given by the Lie bracket formulas for the base $\{X,Y,Z\}$
$$[X,Y]=Z,[Z,X]=[Z,Y]=0.$$
\item the solvable group SOL, whose Lie algebra is given by
$$[X,Y]=0,[Z,X]=X,[Z,Y]=-Y$$
\item the universal cover $\tilde{E}_0(2)$ of the Lie group of proper motions of the euclidean plane  with the Lie algebra
$$[X,Y]=0,[Z,X]=-Y,[Z,Y]=X.$$

\end{enumerate}
The matrices representing classes $[g]$ are
\begin{enumerate}
\item the Heisenberg group NIL
$$
\begin{pmatrix}
\lambda & 0 & 0\\
0 & \lambda & 0\\
0 & 0 &1
\end{pmatrix}, \lambda>0
$$
\item the group SOL
$$
\begin{pmatrix}
1 & 0 & 0\\
0 & 1 & 0\\
0 & 0 & \nu
\end{pmatrix}\;(\mathrm{I}),\qquad
\begin{pmatrix}
1 & 1 & 0\\
1 & \mu & 0\\
0 & 0 & \nu
\end{pmatrix}\;(\mathrm{II}),\quad \mu>1,\ \nu>0
$$
\item the solvable Lie group $\tilde{E}_0(2)$
$$
\begin{pmatrix}
1 & 0  & 0\\
0 & \mu & 0\\
0 & 0 & \nu
\end{pmatrix}, 0<\mu \leq 1,\nu >0.
$$
 
\end{enumerate} 
In our approach we also need an explicit classification of $\mathrm{Aut}(\mathfrak{g})$ from \cite{KHL}.
\begin{enumerate}
\item For the Heisenberg algebra $\mathfrak{g}$ the group $\mathrm{Aut}(\mathfrak{g})$ is isomorphic to a matrix group
consiting of matrices
$$
\begin{pmatrix}
a & c & 0\\
b & d & 0\\
* & * & ad-bc\\
\end{pmatrix}, a,b,c,d\in\mathbb{R},ad-bc\not=0.
$$
\item For the Lie algebra SOL the group $\mathrm{Aut}(\mathfrak{g})$ is isomorphic to the group consisting of matrices
$$
S=\begin{pmatrix}
x_1 & 0 & x_3\\
0 &  y_2 & y_3\\
0 & 0 & 1\\
\end{pmatrix}, x_1,x_3,y_2,y_3\in\mathbb{R},\,\text{or}\, 
S\cdot
\begin{pmatrix}
0 & 1 & 0\\
1 & 0 & 0\\
0 & 0 & -1\\
\end{pmatrix}
$$
\item For the unimodular solvable Lie algebra $\mathfrak{g}=\mathbb{R}^2\rtimes\mathfrak{so}(2)$ the automorphism group $\mathrm{Aut}(\mathfrak{g})$ consists of matrices
$$
\begin{pmatrix}
a & b & \gamma\\
-b & a & \delta\\
0 & 0 & 1\\
\end{pmatrix}, (a,b)\not=(0,0), a,b, \gamma,\delta\in\mathbb{R}.
$$
\end{enumerate}
  
\subsection{Calculation of Ricci curvatures and sectional curvatures in the 3-dimensional case}

 By \cite{M}, Theorem 4.3,  on every 3-dimensional unimodular Lie group, an orthonormal basis $\{Y_1,Y_2,Y_3\}$ with \begin{align*}
    \begin{matrix}
        [Y_2,Y_3] = a_1Y_1, & [Y_3,Y_1] = a_2Y_2,& [Y_1,Y_2] = a_3Y_3
    \end{matrix}
\end{align*}
can be found. We will call such base a \textit{Milnor frame}. In such base, the Ricci tensor is diagonal: indeed, in unimodular situation,
\begin{align*}
    \text{Ric}(X,Y) = &-\dfrac{1}{2}\sum_{k}\langle [X,X_k],[Y,X_k] \rangle \\
    &+ \dfrac{1}{4}\sum_{k,l}\langle[X_k,X_l],X\rangle\langle[X_k,X_l],Y\rangle - \dfrac{1}{2}B(X,Y).
\end{align*}
Here $B(X,Y)$ is the Killing form. Consider $\text{Ric}(Y_1,Y_2)$. The first term vanishes: for $k=1,2$ it follows from anti-symmetry of the Lie bracket, and for $k=3$ we obtain \begin{align*}
    \langle [Y_1,Y_3],[Y_2,Y_3]\rangle = \langle -a_2Y_2,a_1Y_1\rangle = 0.
\end{align*} For the second term, notice that each bracket $[Y_k,Y_l]$ is a multiple of exactly one basis vector. In particular, at most one of the factors $\langle[Y_k,Y_l],Y_1\rangle$, $\langle[Y_k,Y_l],Y_2\rangle$ is nonzero. \\~\\
Similarly, we show that $B(Y_k,Y_l) = 0$ for $k \neq l$. Consider $B(Y_1,Y_2)$. In order to compute it, we calculate $\text{ad}_{Y_2}Y_1 = -a_3Y_3$, so  $\text{ad}_{Y_1}(-a_3Y_3) = a_2a_3Y_2$. Therefore $B(Y_1,Y_2) = \mathrm{tr}(\mathrm{ad}_{Y_1}\circ\mathrm{ad}_{Y_2}) = 0$.\\~\\
The other cases follow by taking any cyclic permutation of the basic vector fields. Now we just need to compute the diagonal entries $\text{Ric}(Y_i,Y_i)$; again, we can just compute $\text{Ric}(Y_1,Y_1)$ as the other cases will follow by taking a cyclic permutation. \\~\\
We will use the same formula as before, which in this case simplifies to $$\text{Ric}(Y_1,Y_1) = -\dfrac{1}{2}\sum_{k}||[Y_1,Y_k]||^2 + \dfrac{1}{4}\sum_{k,l}\langle[Y_k,Y_l],Y_1\rangle^2 - \dfrac{1}{2}B(Y_1,Y_1).$$
For the first term, the only contributing factors are $[Y_1,Y_2] = a_3Y_3$ and $[Y_1,Y_3] = a_2Y_2$. It therefore evaluates to $-\dfrac{1}{2}\left(a_2^2 + a_3^2\right)$. \\~\\ For the second term, only the brackets $[Y_2,Y_3] = a_1Y_1$ and $[Y_3,Y_2] = -a_1Y_1$ are going to contribute; each one contributes with $a_1^2$. Therefore, the second term evaluates to $\dfrac{1}{4}\left(a_1^2+a_1^2\right) = \dfrac{1}{2}a_1^2$. \\~\\
Finally, $B(Y_1,Y_1) = -2a_2a_3$, so the third term evaluates to $a_2a_3$. Combining them together, we obtain that \begin{align*}
    \text{Ric}(Y_1,Y_1) &= -\dfrac{1}{2}(a_2^2 + a_3^2) + \dfrac{1}{2}a_1^2 + a_2a_3 \\ 
    &= \dfrac{1}{2}(a_1^2 - a_2^2 + 2a_2a_3 - a_3^2) = \dfrac{1}{2}(a_1^2 - (a_2-a_3)^2).
\end{align*} 
By cyclic permutations $(i,j,k)$, we obtain the formula for the principal Ricci curvatures
$$\rho_i = \text{Ric}(Y_i,Y_i) = \dfrac{1}{2}(a_i^2 - (a_j - a_k)^2).$$
Denoting the principal Ricci curvatures as $\rho_1,\rho_2,\rho_3$, one may compute the sectional curvatures $K(Y_i,Y_j)$ as \begin{align*}
    K(Y_i,Y_j) = \dfrac{\rho_i + \rho_j - \rho_k}{2}.
\end{align*}
Denote $K(Y_i,Y_j)$ as $K_{ij}$, so that \begin{align*}
    \rho_i &= K_{ij}+ K_{ik},\\
    \rho_j &= K_{ij} + K_{jk}, \\
    \rho_k &= K_{ik} + K_{jk}.
\end{align*}
Thus $\rho_i + \rho_j - \rho_k = 2K_{ij}$ and the formula follows immediately.

\subsection{Invariant Riemannian metrics for SOL}\label{sec:sol-metrics}

\subsubsection{Sectional curvatures}
Let $\{X,Y,Z\}$ be the basis of the Lie algebra of SOL as in the classification found in \cite{KHL}, so
\begin{align*}
    [X,Y]=0,\qquad [Z,X]=X,\qquad [Z,Y]=-Y.
\end{align*}
For both metric types (I) and (II) of Section 4.2, we consider an orthonormal Milnor frame $\{Y_1,Y_2,Y_3\}$ with structure constants
\begin{align*}
    [Y_2,Y_3]=a_1Y_1,\qquad [Y_3,Y_1]=a_2Y_2,\qquad [Y_1,Y_2]=0.
\end{align*}
For metric (I), we take the orthonormal Milnor frame
\begin{align*}
    Y_1=\dfrac{X+Y}{\sqrt{2}},\qquad
    Y_2=\dfrac{-X+Y}{\sqrt{2}},\qquad
    Y_3=\dfrac{Z}{\sqrt{\nu}}.
\end{align*}
Using the bracket relations, we obtain the structure constants
\begin{align*}
    a_1=\dfrac{1}{\sqrt{\nu}},\qquad a_2=-\dfrac{1}{\sqrt{\nu}},\qquad a_3=0.
\end{align*}
For metric (II), we obtain the orthonormal Milnor frame
\begin{align*}
    Y_1&=\dfrac{\sqrt{\mu}X+Y}{\sqrt{2}\sqrt{\mu+\sqrt{\mu}}}, \quad
    Y_2&=\dfrac{-\sqrt{\mu}X+Y}{\sqrt{2}\sqrt{\mu-\sqrt{\mu}}}, \quad
    Y_3&=\dfrac{Z}{\sqrt{\nu}}.
\end{align*}
The bracket relations then give the structure constants
\begin{align*}
    a_1&=\dfrac{\sqrt{\mu+\sqrt{\mu}}}{\sqrt{\nu}\sqrt{\mu-\sqrt{\mu}}},\qquad
    a_2=-\dfrac{\sqrt{\mu-\sqrt{\mu}}}{\sqrt{\nu}\sqrt{\mu+\sqrt{\mu}}},\qquad a_3=0.
\end{align*}
In particular, for metric (I) we have $a_1=-a_2$, whereas for metric (II) we have $a_1>0>a_2$ and $a_1+a_2>0$, since $\mu>1$. We may also assume without loss of generality that $a_1>0>a_2$, as replacing $Y_3$ by $-Y_3$ changes both structure constants to their negatives. In what follows, we do not use the explicit values of $a_1,a_2$ and put $d=a_1-a_2$. \\~\\
Using the formulas for the principal Ricci curvatures obtained in Section 4.3, we calculate
\begin{align*}
    \rho_1 &= \dfrac{1}{2}\left(a_1^2-(a_2-a_3)^2\right)
    =\dfrac{1}{2}(a_1^2-a_2^2)=\dfrac{d(a_1+a_2)}{2},\\
    \rho_2 &= \dfrac{1}{2}\left(a_2^2-(a_3-a_1)^2\right)
    =\dfrac{1}{2}(a_2^2-a_1^2)=-\dfrac{d(a_1+a_2)}{2},\\
    \rho_3 &= \dfrac{1}{2}\left(a_3^2-(a_1-a_2)^2\right)=-\dfrac{d^2}{2}.
\end{align*}
Hence, writing $K_{ij}=K(Y_i,Y_j)$, we obtain
\begin{align*}
    K_{12}&=\dfrac{\rho_1+\rho_2-\rho_3}{2} =\dfrac{1}{4}\left(a_1^2-a_2^2+a_2^2-a_1^2+d^2\right)
    =\dfrac{d^2}{4},\\[0.5em]
    K_{23}&=\dfrac{\rho_2+\rho_3-\rho_1}{2} =\dfrac{1}{4}\left(a_2^2-a_1^2-d^2-a_1^2+a_2^2\right)\\
    &=-\dfrac{d}{4}\left(2(a_1+a_2)+d\right)
    =-\dfrac{d(3a_1+a_2)}{4},\\[0.5em]
    K_{31}&=\dfrac{\rho_3+\rho_1-\rho_2}{2}\\
    &=\dfrac{1}{4}\left(-d^2+a_1^2-a_2^2-a_2^2+a_1^2\right)\\
    &=\dfrac{d}{4}\left(2(a_1+a_2)-d\right)
    =\dfrac{d(a_1+3a_2)}{4}.
\end{align*}

\subsubsection{Weyl sectional curvature}
We compute the terms $U(Y_i,Y_j)$ using Proposition \ref{prop:KN-LC}. For $i\neq j$, the components along $Y_i$ and $Y_j$ vanish, as follows from the bracket relations and orthonormality. We calculate the remaining components:
\begin{align*}
    2\langle U(Y_1,Y_2),Y_3\rangle
    &=\langle Y_1,[Y_3,Y_2]\rangle+\langle[Y_3,Y_1],Y_2\rangle\\
    &=\langle Y_1,-a_1Y_1\rangle+\langle a_2Y_2,Y_2\rangle
    =-d,\\
    2\langle U(Y_1,Y_3),Y_2\rangle
    &=\langle Y_1,[Y_2,Y_3]\rangle+\langle[Y_2,Y_1],Y_3\rangle\\
    &=\langle Y_1,a_1Y_1\rangle=a_1,\\
    2\langle U(Y_2,Y_3),Y_1\rangle
    &=\langle Y_2,[Y_1,Y_3]\rangle+\langle[Y_1,Y_2],Y_3\rangle\\
    &=\langle Y_2,-a_2Y_2\rangle=-a_2.
\end{align*}
Therefore
\begin{align*}
    U(Y_1,Y_2)=-\dfrac{d}{2}Y_3,\qquad
    U(Y_1,Y_3)=\dfrac{a_1}{2}Y_2,\qquad
    U(Y_2,Y_3)=-\dfrac{a_2}{2}Y_1.
\end{align*}
Moreover, $U(Y_i,Y_i)=0$, since $\langle Y_i,[Y_k,Y_i]\rangle=0$ for every $k$.
Using the symmetry of $U$, we find the values of the Nomizu function:
\begin{align*}
    \alpha(Y_1,Y_2)&=-\dfrac{d}{2}Y_3, &
    \alpha(Y_2,Y_1)&=-\dfrac{d}{2}Y_3,\\
    \alpha(Y_1,Y_3)&=\dfrac{-a_2+a_1}{2}Y_2=\dfrac{d}{2}Y_2, &
    \alpha(Y_3,Y_1)&=\dfrac{a_2+a_1}{2}Y_2,\\
    \alpha(Y_2,Y_3)&=\dfrac{a_1-a_2}{2}Y_1=\dfrac{d}{2}Y_1, &
    \alpha(Y_3,Y_2)&=-\dfrac{a_1+a_2}{2}Y_1.
\end{align*}
Also, $\alpha(Y_i,Y_i)=0$ for $i=1,2,3$. \\~\\
Let $E=e_1Y_1+e_2Y_2+e_3Y_3$ be an arbitrary left-invariant vector field defining the Weyl connection. By bilinearity, we obtain
\begin{align*}
    \alpha(Y_1,E)&=e_1\alpha(Y_1,Y_1)+e_2\alpha(Y_1,Y_2)+e_3\alpha(Y_1,Y_3)\\
    &=\dfrac{d}{2}(e_3Y_2-e_2Y_3),\\
    \alpha(Y_2,E)&=e_1\alpha(Y_2,Y_1)+e_2\alpha(Y_2,Y_2)+e_3\alpha(Y_2,Y_3)\\
    &=\dfrac{d}{2}(e_3Y_1-e_1Y_3),\\
    \alpha(Y_3,E)&=e_1\alpha(Y_3,Y_1)+e_2\alpha(Y_3,Y_2)+e_3\alpha(Y_3,Y_3)\\
    &=\dfrac{a_1+a_2}{2}(e_1Y_2-e_2Y_1).
\end{align*}
In particular, $\langle\alpha(Y_i,E),Y_i\rangle=0$ for $i=1,2,3$. We now substitute these values into the formula for the Weyl sectional curvature given by Proposition \ref{prop:hat-k}. For the plane spanned by $Y_1,Y_2$, we obtain
\begin{align*}
    \widehat K(Y_1,Y_2)&=K_{12}+e_1^2+e_2^2-(e_1^2+e_2^2+e_3^2)\\
    &\quad-\langle Y_2,\alpha(Y_2,E)\rangle-\langle\alpha(Y_1,E),Y_1\rangle\\
    &=\dfrac{d^2}{4}-e_3^2.
\end{align*}
Analogously,
\begin{align*}
    \widehat K(Y_2,Y_3)&=-\dfrac{d(3a_1+a_2)}{4}-e_1^2,\\
    \widehat K(Y_3,Y_1)&=\dfrac{d(a_1+3a_2)}{4}-e_2^2.
\end{align*}

\subsection{Heisenberg algebra}\label{sec:nil}
\noindent The standard basis $\{X,Y,Z\}$ of the Lie algebra of the Heisenberg group $\mathrm{NIL}$ satisfies
\begin{align*}
\noindent [X,Y]=Z,
\qquad
[Z,X]=[Z,Y]=0.
\end{align*}

\noindent We consider the orthonormal Milnor frame
\begin{align*}
\noindent
Y_1=\frac{1}{\sqrt{\lambda}}X,
\qquad
Y_2=\frac{1}{\sqrt{\lambda}}Y,
\qquad
Y_3=Z.
\end{align*}

\noindent with the structure constants
\begin{align*}
\noindent
[Y_1,Y_2]=\frac{1}{\lambda}Y_3,
\qquad
[Y_2,Y_3]=0,
\qquad
[Y_3,Y_1]=0.
\end{align*}

\noindent The sectional curvatures are
\begin{align*}
K(Y_1,Y_2)=-\frac{3}{4\lambda^2},
\qquad
K(Y_2,Y_3)=\frac{1}{4\lambda^2},
\qquad
K(Y_3,Y_1)=\frac{1}{4\lambda^2}.
\end{align*}

\noindent The values of the terms $U(Y_1,Y_j),\; i=1,2,3$ are
\begin{align*}
U(Y_1,Y_2)=0,
\qquad
U(Y_1,Y_3)=-\frac{1}{2\lambda}Y_2,
\qquad
U(Y_2,Y_3)=\frac{1}{2\lambda}Y_1.
\end{align*}

\noindent Also,
\begin{align*}
U(Y_i,Y_i)=0
\qquad
\text{for } i=1,2,3.
\end{align*}

\noindent Thus the relevant values of the Nomizu function $\alpha$ are
\begin{align*}
\begin{matrix*}[l]
    \alpha(Y_1,Y_2)=\dfrac{1}{2\lambda}Y_3, &&
\alpha(Y_1,Y_3)=-\dfrac{1}{2\lambda}Y_2, && \alpha(Y_2,Y_3)=\dfrac{1}{2\lambda}Y_1, \\[1em]
\alpha(Y_2,Y_1)=-\dfrac{1}{2\lambda}Y_3, &&
\alpha(Y_3,Y_1)=-\dfrac{1}{2\lambda}Y_2, &&
\alpha(Y_3,Y_2)=\dfrac{1}{2\lambda}Y_1.
\end{matrix*}
\end{align*}

\noindent For any left-invariant vector field $E = e_1Y_1 + e_2Y_2 + e_3Y_3,$ the values of $\alpha(Y_i,E)$ are given by
\begin{align*}
\alpha(Y_1,E) &= \dfrac{e_2}{2\lambda}Y_3 - \dfrac{e_3}{2\lambda}Y_2, \\
\alpha(Y_2,E) &= \dfrac{e_3}{2\lambda}Y_1 -\dfrac{e_1}{2\lambda}Y_3 , \\
\alpha(Y_3,E) &=  \dfrac{e_2}{2\lambda}Y_1-\dfrac{e_1}{2\lambda}Y_2.
\end{align*}

\noindent In particular,
\begin{align*}
\noindent
\langle \alpha(Y_i,E),Y_i\rangle=0
\qquad
\text{for } i=1,2,3.
\end{align*}

\noindent Hence
\begin{align*}
\widehat{K}(Y_1,Y_2) &=
-\frac{3}{4\lambda^2}-e_3^2. \\
\widehat{K}(Y_2,Y_3) &=
\frac{1}{4\lambda^2} - e_1^2, \\
\widehat{K}(Y_3,Y_1) &=
\frac{1}{4\lambda^2} - e_2^2.
\end{align*}

\subsection{Completion of proof}

\subsubsection{Group SOL}
Let $N=xY_1+yY_2+zY_3$ be a unit normal to a plane $\Pi$ and let $E=e_1Y_1+e_2Y_2+e_3Y_3$ be an arbitrary left-invariant vector field defining the Weyl connection $\hat\nabla$. Note that \begin{align*}
     \langle X,E\rangle^2+\langle Y,E\rangle^2-\langle E,E\rangle&=-\langle E,N\rangle^2
\end{align*} and that by unimodularity of SOL, the vector field $E$ is divergence-free, hence \begin{align*}
    \langle\nabla_XE,X\rangle+\langle\nabla_YE,Y\rangle&=-\langle\nabla_NE,N\rangle.
\end{align*} If $X,Y$ form an orthonormal basis of of a plane $\Pi$, these identities together with Proposition \ref{prop:hat-k} yield
\begin{align}\label{eq:weyl}
    \widehat K_E(\Pi)=K(\Pi)+\langle\nabla_NE,N\rangle-\langle E,N\rangle^2.
\end{align}
By the diagonality of the Ricci tensor in the Milnor frame, we obtain the diagonality of the sectional curvature operator in the induced basis $\{Y_2\wedge Y_3, Y_3\wedge Y_1,Y_1\wedge Y_2\}$. Therefore
\begin{align}\label{eq:K}
    K(\Pi)=K_{23}x^2+K_{31}y^2+K_{12}z^2.
\end{align}
We first calculate the quadratic form corresponding to the second term of \eqref{eq:weyl}. Recall that $\nabla_{Y_i}E=\alpha(Y_i,E)$ for the left-invariant vector fields under consideration. By the calculations of Section \ref{sec:sol-metrics},
\begin{align*}
&\nabla_{Y_1}E=\frac d2(e_3Y_2-e_2Y_3), \qquad
\nabla_{Y_2}E=\frac d2(e_3Y_1-e_1Y_3), \\
&\nabla_{Y_3}E=\frac{a_1+a_2}2(e_1Y_2-e_2Y_1).
\end{align*}
Consequently,
\begin{align*}
\nabla_NE
&=\left(\frac{dye_3}{2}-\frac{(a_1+a_2)ze_2}{2}\right)Y_1 +\left(\frac{dxe_3}{2}+\frac{(a_1+a_2)ze_1}{2}\right)Y_2\\
&\qquad+\left(-\frac{dxe_2}{2}-\frac{dye_1}{2}\right)Y_3.
\end{align*}
and therefore
\begin{align*}
\langle\nabla_NE,N\rangle
&=\left(\frac{dye_3}{2}-\frac{(a_1+a_2)ze_2}{2}\right)x+\left(\frac{dxe_3}{2}+\frac{(a_1+a_2)ze_1}{2}\right)y\\
&\qquad+\left(-\frac{dxe_2}{2}-\frac{dye_1}{2}\right)z \\ &=de_3xy-\frac{a_1+a_2+d}{2}e_2xz  +\frac{a_1+a_2-d}{2}e_1yz.
\end{align*}
Since $a_1+a_2+d=2a_1$ and $a_1+a_2-d=2a_2$, it simplifies to
\begin{align}\label{eq:derivative}
\langle\nabla_NE,N\rangle=de_3xy-a_1e_2xz+a_2e_1yz.
\end{align}
Furthermore,
\begin{align}\label{eq:perp}
\langle E,N\rangle^2&=(e_1x+e_2y+e_3z)^2\\ \nonumber
&=e_1^2x^2+e_2^2y^2+e_3^2z^2+2e_1e_2xy+2e_1e_3xz+2e_2e_3yz.
\end{align}
Substituting \eqref{eq:K}, \eqref{eq:derivative} and \eqref{eq:perp} into \eqref{eq:weyl}, we obtain
\begin{align*}
\widehat K_E(\Pi)
={}&(K_{23}-e_1^2)x^2+(K_{31}-e_2^2)y^2+(K_{12}-e_3^2)z^2\\
&+(de_3-2e_1e_2)xy-(a_1e_2+2e_1e_3)xz+(a_2e_1-2e_2e_3)yz.
\end{align*}
Writing $\widehat K_E(\Pi)=(x\;y\;z)M_E(x\;y\;z)^T$ and substituting the sectional curvatures from Section \ref{sec:sol-metrics}, we obtain that the matrix $M_E$ is given by
\begin{align}\label{eq:matrix}
M_E=\frac14\begin{pmatrix}
-d(3a_1+a_2)-4e_1^2&2de_3-4e_1e_2&-2a_1e_2-4e_1e_3\\[3pt]
2de_3-4e_1e_2&d(a_1+3a_2)-4e_2^2&2a_2e_1-4e_2e_3\\[3pt]
-2a_1e_2-4e_1e_3&2a_2e_1-4e_2e_3&d^2-4e_3^2
\end{pmatrix}.
\end{align}
Consider the expression for Weyl sectional curvature as a homogeneous quadratic form $\widehat{K}_E(\Pi) = Q_E(N)$ defined for $N=xY_1+yY_2+zY_3$ by
\begin{align*}
Q_E(N)=(x\;y\;z)M_E(x\;y\;z)^T.
\end{align*}
Therefore $\widehat K_E(\Pi)\leq0\text{ for all }\Pi$ if and only if $M_E$ is negative semidefinite. Assume without loss of generality that $a_1>0>a_2$ and for clarity define
\begin{align*}
h=\frac{a_1-a_2}{2},\qquad t=\frac{a_1+a_2}{2},
\end{align*}
so that
\begin{align*}
K_{23}=-h^2-2ht,\qquad K_{31}=-h^2+2ht,\qquad K_{12}=h^2.
\end{align*}
Moreover, $h > 0$ and since $h^2-t^2 = -a_1a_2 >0$ we obtain $|t| < h.$ If $M_E$ is negative semidefinite, then its last diagonal entry must be non-positive:
\begin{align*}
(M_E)_{33}=h^2-e_3^2\leq0,
\end{align*}
so $|e_3| \geq h$. We will now show that the strict inequality is impossible by constructing an explicit example. Introduce the orthonormal basis
\begin{align*}
U=\frac{Y_1+Y_2}{\sqrt2},\qquad V=\frac{Y_1-Y_2}{\sqrt2},\qquad W=Y_3,
\end{align*}
in which $E=pU+qV+rW$, where
\begin{align*}
p=\frac{e_1+e_2}{\sqrt2},\qquad q=\frac{e_1-e_2}{\sqrt2},\qquad r=e_3.
\end{align*}
We may also assume without loss of generality that $r\geq h$: replacing $Y_1,Y_3$ by $-Y_1,-Y_3$ preserves the bracket relations and reverses the sign of $e_3$. In the new basis, the matrix of $Q_E$ is $\widetilde M_E=P^TM_EP$, where
\begin{align*}
P=\frac{1}{\sqrt{2}}\begin{pmatrix}1&1&0\\1&-1&0\\0&0&\sqrt{2}\end{pmatrix}.
\end{align*}
A direct substitution into \eqref{eq:matrix} yields
\begin{align}\label{eq:newmatrix}
\widetilde M_E=\begin{pmatrix}
h(r-h)-p^2&-2ht-pq&-\frac{hp-tq}{2}-pr\\[5pt]
-2ht-pq&-h(r+h)-q^2&\frac{hq-tp}{2}-qr\\[5pt]
-\frac{hp-tq}{2}-pr&\frac{hq-tp}{2}-qr&h^2-r^2
\end{pmatrix}.
\end{align}
Since $P$ is orthogonal, $M_E$ is semidefinite if and only if $\widetilde M_E$ is. Suppose that $r = e_3 >h$ and $\widetilde M_E$ is negative semidefinite. Consider the nonzero vector
\begin{align*}
N_1=rU-pW.
\end{align*}
Note that $\langle E,N_1\rangle^2=(pr-rp)^2=0$. Using \eqref{eq:newmatrix}, we calculate
\begin{align*}
Q_E(N_1)
&=r^2\big(h(r-h)-p^2\big)-2rp\left(-\frac{hp-tq}{2}-pr\right)+p^2(h^2-r^2)\\
&=hr^2(r-h)+h(r+h)p^2-trpq \leq 0
\end{align*}
by negative semidefiniteness.
Dividing the last inequality by $r>0$, we obtain
\begin{align*}
tpq\geq hr(r-h)+h\left(1+\frac hr\right)p^2.
\end{align*}
Since $r>h$, the term $hr(r-h)$ is positive, so in particular
\begin{align}\label{eq:tpq}
tpq>hp^2.
\end{align}
It follows that $p\neq0$ and $tpq>0$. As $|t|<h$, we have
\begin{align*}
h|pq|>|tpq|=tpq>hp^2.
\end{align*}
Dividing by $h|p|>0$ gives $|q|>|p|$, or equivalently
\begin{align}\label{eq:pq}
q^2>p^2.
\end{align}
Now consider another nonzero vector $N_2=qU-pV$. Again, $N_2\perp E$, and we compute
\begin{align*}
Q_E(N_2)
&=q^2\left(h(r-h)-p^2\right)-2pq(-2ht-pq) +p^2\left(-h(r+h)-q^2\right)\\
&=h(r-h)(q^2-p^2)-2h^2p^2+4htpq.
\end{align*}
By \eqref{eq:tpq} and \eqref{eq:pq}, we obtain
\begin{align*}
Q_E(N_2)&>h(r-h)(q^2-p^2)-2h^2p^2+4h^2p^2\\
&=h(r-h)(q^2-p^2)+2h^2p^2 >0,
\end{align*}
which contradicts the negative semidefiniteness of $\widetilde M_E$. Therefore $r>h$ is impossible, and returning to the original notation, we have proved that
\begin{align}\label{eq:e3}
|e_3|=\frac{|a_1-a_2|}{2}.
\end{align}
Next, we show that $e_1=e_2=0$. By \eqref{eq:e3}, the last diagonal entry of $M_E$ vanishes. Note that whenever a real symmetric negative semidefinite matrix $M=(m_{ij})$ has $m_{jj}=0$, its entire $j$-th row and column must vanish. Indeed, its principal minors of order two must be non-negative, and thus
\begin{align*}
\det\begin{pmatrix}m_{ii}&m_{ij}\\m_{ij}&0\end{pmatrix}=-m_{ij}^2\geq0.
\end{align*}
Applying this observation to \eqref{eq:matrix}, we obtain the necessary conditions
\begin{align*}
-\frac{a_1e_2}{2}-e_1e_3=0,\qquad \frac{a_2e_1}{2}-e_2e_3=0.
\end{align*}
Equivalently,
\begin{align*}
\begin{pmatrix}2e_3&a_1\\a_2&-2e_3\end{pmatrix}
\begin{pmatrix}e_1\\e_2\end{pmatrix}=\begin{pmatrix}0\\0\end{pmatrix}.
\end{align*}
The determinant of the coefficient matrix is
\begin{align*}
-4e_3^2-a_1a_2&=-d^2-a_1a_2\\
&=-(a_1-a_2)^2-a_1a_2\\
&=-a_1^2+a_1a_2-a_2^2<0
\end{align*}
since $a_1a_2<0$. Therefore $e_1= e_2 = 0.$
We may therefore only consider Weyl connections defined by vector fields of the form $E=e_3Y_3$. It remains to determine for which values of $a_1,a_2$ the matrix
\begin{align*}
M_E=\begin{pmatrix}
-\frac{d(3a_1+a_2)}4&\frac{de_3}2&0\\[5pt]
\frac{de_3}2&\frac{d(a_1+3a_2)}4&0\\[5pt]
0&0&0
\end{pmatrix}
\end{align*}
is negative semidefinite, where $e_3^2=d^2/4$. The remaining principal minor of order two is necessarily non-negative. However,
\begin{align*}
\det\begin{pmatrix}-\frac{d(3a_1+a_2)}4&\frac{de_3}2\\[3pt]
\frac{de_3}2&\frac{d(a_1+3a_2)}4\end{pmatrix}
&=-\frac{d^2}{16}(3a_1+a_2)(a_1+3a_2)-\frac{d^2e_3^2}{4}
\\ &=-\frac{d^2}{4}(a_1+a_2)^2.
\end{align*}
Since $d\neq0$, this expression is non-negative if and only if $a_1+a_2 = 0$. In this case
\begin{align*}
K_{23}=K_{31}=-a^2,\qquad K_{12}=a^2.
\end{align*}
and the matrix $M_E$ is given by
\begin{align*}
M_E=a^2\begin{pmatrix}-1&\pm1&0\\\pm1&-1&0\\0&0&0\end{pmatrix},
\end{align*}
which is negative semidefinite. Therefore $M_E$ is negative semidefinite if and only if
\begin{align*}
    a_1+a_2=0,\qquad e_1=e_2=0,\qquad e_3=\pm\dfrac{a_1-a_2}{2}.
\end{align*}
By Section \ref{sec:sol-metrics}, the condition $a_1+a_2=0$ holds precisely for metric (I), while for metric (II)
\begin{align*}
    a_1+a_2&=\dfrac{(\mu+\sqrt{\mu})-(\mu-\sqrt{\mu})}
    {\sqrt{\nu}\sqrt{(\mu+\sqrt{\mu})(\mu-\sqrt{\mu})}}\\
    &=\dfrac{2}{\sqrt{\nu(\mu-1)}}>0.
\end{align*}
Thus the only left-invariant Weyl connections of non-positive sectional Weyl curvature on SOL occur for metrics of type (I) and are determined by
\begin{align*}
    E=\pm\dfrac{1}{\sqrt{\nu}}Y_3.
\end{align*}
These are precisely the isolated non-positive examples found by Tereszkiewicz and Wojtkowski in \cite{TW}. The classification for metrics of type (I) also follows from their results concerning symmetric extensions of abelian Lie algebras, particularly Theorem 5.3 and Corollary 5.4. Our calculation shows, moreover, that the symmetry of the corresponding adjoint operator is a necessary condition for the existence of a non-positive left-invariant Weyl connection on SOL. In particular, no metric of type (II) admits such a connection. \\~\\ 
Note that these connections are not SNP: for any stretching factor $\gamma>1$, the field $\gamma E$ no longer satisfies the necessary condition $|e_3|=d/2$. 

\subsubsection{The group NIL}
The left-invariant metric on NIL is given by
\[
\begin{pmatrix}
\lambda & 0 & 0\\
0 & \lambda & 0\\
0 & 0 & 1
\end{pmatrix},
\qquad \lambda>0.
\]
We use the orthonormal Milnor frame
\[
Y_1=\frac{1}{\sqrt{\lambda}}X,\qquad
Y_2=\frac{1}{\sqrt{\lambda}}Y,\qquad
Y_3=Z.
\]
with structure constants
\[
[Y_1,Y_2]=\frac{1}{\lambda}Y_3,\qquad
[Y_2,Y_3]=0,\qquad
[Y_3,Y_1]=0.
\]
Define $a_3= \dfrac{1}{\lambda}$. By the calculations from Section \ref{sec:nil}, the sectional curvatures are
\[
K(Y_1,Y_2)=-\frac{3}{4\lambda^2},
\qquad
K(Y_2,Y_3)=\frac{1}{4\lambda^2},
\qquad
K(Y_3,Y_1)=\frac{1}{4\lambda^2}.
\]
Moreover, by the classification of $\mathrm{Aut}(\mathfrak{g})$ from \cite{KHL}, an automorphism of NIL is of form
\[
\begin{pmatrix}
a & c & 0\\
b & d & 0\\
* & * & ad-bc
\end{pmatrix},
\qquad ad-bc\neq 0.
\]
Therefore, with respect to the Milnor frame, we may write \begin{align*}
    \varphi(Y_1)&=aY_1+bY_2+rY_3,\\
    \varphi(Y_2)&=cY_1+dY_2+sY_3, \\
\varphi(Y_3)&=(ad-bc)Y_3.
\end{align*}
Assume that $\varphi(Y_i)$ are scaled such that $\varphi(Y_i) = 1$. This yields
$\|\varphi(Y_3)\|^2=(ad-bc)^2,$ and thus $(ad-bc)^2 = 1$. By restricting to the connected component of the identity, we may assume $ad - bc = 1$. \\~\\
We may also assume that all the planes are spanned by orthogonal basis; consider the base $\{\varphi(Y_1),\;\varphi(Y_2),\; \varphi(Y_3)\}$. Since $
\langle \varphi(Y_1),\varphi(Y_3)\rangle=r(ad-bc)=r$ and $\langle \varphi(Y_2),\varphi(Y_3)\rangle=s(ad-bc)=s,$ we obtain that $r=s=0$ by orthogonality. Thus
\begin{align*}
\varphi(Y_1)=aY_1+bY_2,
\qquad
\varphi(Y_2)=cY_1+dY_2,
\qquad
\varphi(Y_3)=Y_3.
\end{align*}
and the matrix $\begin{pmatrix}
a & c\\
b & d
\end{pmatrix} $ lies in SO(2). We can therefore write \begin{align*}
\varphi(Y_1)&=\cos\theta\,Y_1+\sin\theta\,Y_2,\\
\varphi(Y_2)&=-\sin\theta\,Y_1+\cos\theta\,Y_2.
\end{align*}
Let $E=e_1Y_1+e_2Y_2+e_3Y_3.$ and consider the Weyl sectional curvatures of the planes spanned by
\(\varphi(Y_i)\) and \(\varphi(Y_j)\). Since the vectors
\(\varphi(Y_1),\varphi(Y_2),\varphi(Y_3)\) are orthonormal, the Weyl sectional curvature is given by
\begin{align*}
\widehat K_E(X,Y)&=K(X,Y)+\langle Y,E\rangle^2+\langle X,E\rangle^2-\langle E,E\rangle\\
&\quad-\langle Y,\alpha(Y,E)\rangle-\langle \alpha(X,E),X\rangle.
\end{align*}
Note that the linear terms vanish. Indeed, for any horizontal unit vector of form $U = uY_1 +vY_2$ one has \begin{align*}
    \langle \alpha(U,E),U\rangle=0, \qquad \langle \alpha(Y_3,E),Y_3\rangle=0.
\end{align*}
Hence the formula for the Weyl sectional curvature reduces to
\[\widehat K_E(X,Y)=K(X,Y)+\langle X,E\rangle^2+\langle Y,E\rangle^2-\langle E,E\rangle.\]
Equivalently, if $N$ is the unit normal vector to the plane
$\mathrm{span}(X,Y)$, then
$$
\widehat K_E(X,Y)
=
K(X,Y)-\langle E,N\rangle^2.
$$
The Weyl sectional curvature for the plane $\mathrm{span}(\varphi(Y_1),\varphi(Y_2))$ is always negative. Indeed,
$$\mathrm{span} \varphi(Y_1),\varphi(Y_2)) =\mathrm{span}(Y_1,Y_2).$$
and its unit normal vector is $Y_3$. Therefore
$$\widehat K_E(\varphi(Y_1),\varphi(Y_2))= K(Y_1,Y_2)-e_3^2 = -\frac{3}{4\lambda^2}-e_3^2 < 0.$$
For the plane $\mathrm{span}(\varphi(Y_2),\varphi(Y_3)),$ the unit normal vector is \(\varphi(Y_1)\). Hence 
$$
\widehat K_E(\varphi(Y_2),\varphi(Y_3)) = K(Y_2,Y_3)-\langle E,\varphi(Y_1)\rangle^2.
$$
Since $\varphi(Y_1)=\cos\theta\,Y_1+\sin\theta\,Y_2,$ we obtain that $\langle E,\varphi(Y_1)\rangle=e_1\cos\theta+e_2\sin\theta$, and therefore
\begin{align}\label{eq:nil1}
\widehat K_E(\varphi(Y_2),\varphi(Y_3))=\frac{1}{4\lambda^2}-(e_1\cos\theta+e_2\sin\theta)^2.
\end{align}
For the plane $\mathrm{span}(\varphi(Y_3),\varphi(Y_1)),$
the unit normal vector is \(\varphi(Y_2)\). Hence 
\begin{align*}
    \widehat K_E(\varphi(Y_3),\varphi(Y_1))=K(Y_3,Y_1)-\langle E,\varphi(Y_2)\rangle^2.
\end{align*}
Since $\varphi(Y_2)=-\sin\theta\,Y_1+\cos\theta\,Y_2,$ we get that $\langle E,\varphi(Y_2)\rangle=-e_1\sin\theta+e_2\cos\theta$.
Therefore
\begin{align}\label{eq:nil2}
\widehat K_E(\varphi(Y_3),\varphi(Y_1)) = \frac{1}{4\lambda^2} - (-e_1\sin\theta+e_2\cos\theta)^2.
\end{align}
We will show that both Weyl sectional curvatures (\ref{eq:nil1}) and (\ref{eq:nil2}) cannot be simulteanously non-positive for any automorphism $\varphi$, or equivalently for every angle $\theta$. \\~\\
Let $E_H=e_1Y_1+e_2Y_2$. If \(E_H=0\), i.e. $e_1=e_2=0$, then
$$
\widehat K_E(\varphi(Y_2),\varphi(Y_3))=\widehat K_E(\varphi(Y_3),\varphi(Y_1))=\frac{1}{4\lambda^2}>0.
$$
Now assume $E_H\neq 0$ and choose $\theta$ such that $\varphi(Y_1)\perp E_H$, i.e. $\langle E,\varphi(Y_1)\rangle=0.$ Consequently,
\[
\widehat K_E(\varphi(Y_2),\varphi(Y_3))
=
\frac{1}{4\lambda^2}>0.
\]
Therefore for every left-invariant vector field \(E\), there exists an
automorphic isometry \(\varphi\) such that
$$
\widehat K_E(\varphi(Y_i),\varphi(Y_j))>0
$$
for some \(i\neq j\). Hence NIL admits no non-positive left-invariant Weyl
connections.
\subsection{Euclidean group}
 To illustrate our method we attach calculatiuons showing that $\tilde{E}_0(2)$ admits SNP connections. 

\noindent The Lie algebra
$\mathbb{R}^2\rtimes\mathfrak{so}(2)$ of the universal cover of $\mathbb{R}^2\rtimes SO(2)$ admits the basis $\{X,Y,Z\}$ satisfying
\begin{align*}
[X,Y]=0,\qquad [Z,X]=-Y,\qquad [Z,Y]=X.
\end{align*}
Consider the orthonormal Milnor frame given by
\begin{align*}
Y_1=X,
\qquad
Y_2=\frac{1}{\sqrt{\mu}}Y,
\qquad
Y_3=\frac{1}{\sqrt{\nu}}Z
\end{align*}
with the structure constants
\begin{align*}
[Y_2,Y_3]
=
-\frac{1}{\sqrt{\mu\nu}}Y_1,
\qquad
[Y_3,Y_1]
=
-\frac{\sqrt{\mu}}{\sqrt{\nu}}Y_2,
\qquad
[Y_1,Y_2]=0.
\end{align*}
\noindent The values of the terms $U(Y_1,Y_j),\; i=1,2,3$ are
\begin{align*}
U(Y_1,Y_2)= \frac{1-\mu}{2\sqrt{\mu\nu}}Y_3,
\qquad
U(Y_1,Y_3)=-\frac{1}{2\sqrt{\mu\nu}}Y_2, \qquad 
U(Y_2,Y_3)= \frac{\mu}{2\sqrt{\mu\nu}}Y_1.
\end{align*}
with \begin{align*}
    U(Y_i,Y_i) = 0 \qquad \text{for}\; i=1,2,3.
\end{align*}
We then find the values $\alpha(Y_i,Y_j)$ of the Nomizu function:
\begin{align*}
\begin{matrix*}[l]
\alpha(Y_1,Y_2)=\dfrac{1-\mu}{2\sqrt{\mu\nu}}Y_3, &&
\alpha(Y_1,Y_3)=-\dfrac{1-\mu}{2\sqrt{\mu\nu}}Y_2, &&
\alpha(Y_2,Y_3)=-\dfrac{1-\mu}{2\sqrt{\mu\nu}}Y_1. \\[1em]
\alpha(Y_2,Y_1)=\dfrac{1-\mu}{2\sqrt{\mu\nu}}Y_3, &&
\alpha(Y_3,Y_1)=-\dfrac{1+\mu}{2\sqrt{\mu\nu}}Y_2, &&
\alpha(Y_3,Y_2)=\dfrac{1+\mu}{2\sqrt{\mu\nu}}Y_1.
\end{matrix*}
\end{align*}
Taking the left-invariant vector field $E = e_1Y_1 + e_2Y_2 + e_3Y_3$, we obtain
\begin{align*}
\alpha(Y_1,E) &= 
\frac{1-\mu}{2\sqrt{\mu\nu}}
\left(e_2Y_3-e_3Y_2\right), \\
\alpha(Y_2,E) &=
\frac{1-\mu}{2\sqrt{\mu\nu}}
\left(e_1Y_3-e_3Y_1\right), \\
\alpha(Y_3,E) &=
\frac{1+\mu}{2\sqrt{\mu\nu}}
\left(-e_1Y_2+e_2Y_1\right).
\end{align*}

\noindent Therefore
\begin{align*}
\langle \alpha(Y_i,E),Y_i\rangle=0
\qquad
\text{for } i=1,2,3.
\end{align*}
Finally, we obtain
\begin{align*}
\widehat{K}(Y_1,Y_2) &= 
\frac{(1-\mu)^2}{4\mu\nu} - e_3^2, \\
\widehat{K}(Y_2,Y_3) &= 
-\frac{(1-\mu)(3+\mu)}{4\mu\nu} -e_1^2, \\
\widehat{K}(Y_3,Y_1) &=
\frac{(1-\mu)(1+3\mu)}{4\mu\nu} - e_2^2.
\end{align*}

Clearly, this yields an SNP Weyl connection, as in \cite{TW}.

\end{document}